\documentclass[12pt, a4paper]{article}

\usepackage[T1]{fontenc}
\usepackage[utf8]{inputenc}
\usepackage[english]{babel}

\usepackage[autostyle]{csquotes}
\usepackage{mlmodern}

\usepackage[top=3.3cm, bottom=3.3cm, left=3cm, right=2.5cm]{geometry}
\usepackage{fancyhdr}
\usepackage{setspace}
\usepackage{amsfonts,amssymb,amsthm,amsmath,amscd,mathtools}
\usepackage{dsfont}
\usepackage{bm}
\usepackage{esint}
\usepackage{enumitem}

\usepackage{subcaption,booktabs}

\usepackage[table]{xcolor}
\usepackage{graphicx}
\usepackage{tikz}
\usetikzlibrary{positioning, shapes.geometric, arrows.meta}

\usepackage{natbib}
\usepackage{newtxtext}
\usepackage[subscriptcorrection]{newtxmath}
\usepackage[plain,noend]{algorithm2e}
\usepackage[normalem]{ulem}
\usepackage{comment}

\usepackage[citecolor=blue]{hyperref}

\theoremstyle{definition}
\newtheorem{theorem}{Theorem}[section]

\newtheorem{proposition}[theorem]{Proposition}
\newtheorem{corollary}[theorem]{Corollary}
\newtheorem{definition}[theorem]{Definition}

\newtheorem{remark}[theorem]{Remark}

\newtheorem{condition}[theorem]{Condition}

\usepackage{titlesec}
\titleformat{\section}
{\normalfont\Large\bfseries}
{\thesection}{1em}{}
\titleformat{\subsection}
{\normalfont\large\bfseries}
{\thesubsection}{1em}{}
\titleformat{\subsubsection}
{\normalfont\normalsize\bfseries}
{\thesubsubsection}{1em}{}

\title{\textbf{Beyond Degree: Rooted Motif Signatures for Latent Position Identifiability in Graphon Models}}
\author{
	\textbf{Roland B. Sogan}\\
	\small Sorbonne Université, Université Paris Cité, CNRS,\\
	\small Laboratoire de Probabilités, Statistique et Modélisation,\\
	\small \href{mailto:roland-boniface.sogan@sorbonne-universite.fr}{\texttt{roland-boniface.sogan@sorbonne-universite.fr}}
	\and
	\textbf{Tabea Rebafka}\\
	\small MIA AgroParisTech, INRAE, Université Paris-Saclay,\\
	\small \href{mailto:tabea.rebafka@agroparistech.fr}{\texttt{tabea.rebafka@agroparistech.fr}}
	\and
}
\date{}
\begin{document}
	
	\maketitle
	%\tableofcontents
	
	\begin{abstract}

		Graphon estimation requires structural assumptions to address its intrinsic non-identifiability. A standard approach is degree-based identifiability, where the degree function is assumed to be strictly monotonic. This assumption is rather restrictive and fails for graphons with constant or non-injective degree function, even when distinct latent positions have different connectivity profiles. 
		In this paper, we introduce \emph{rooted motif signatures} as higher-order node-level representations for graphons. They extend the degree function by recording, at each latent position, the densities of rooted motifs such as triangles, cycles, paths, and other local subgraph patterns. We study the extent to which these signatures can distinguish latent positions beyond degree information. For generic finite-rank graphons,
		we prove that suitable rooted motif signatures determine the connectivity profiles of latent positions. We also explain why such a property cannot hold for arbitrary graphons without additional assumptions, since different latent positions may have identical rooted motif signatures.
		On the statistical side, we define empirical rooted motif signatures from a single observed graph and prove uniform concentration bounds for these estimators. % As aconsequence, the motif-induced distances between nodes are consistently estimated. 
		Simulation experiments illustrate that rooted motif signatures can reveal latent structure in settings where degree-based representations are uninformative, including
		graphons with constant or non-injective degree functions and stochastic block models with
		equal block degrees.
		%	We also explain why such identifiability cannot hold for arbitrary graphons without additional assumptions, since symmetries of the latent space may make distinct connectivity profiles indistinguishable by all rooted motif statistics.
	\end{abstract}

	\section{Introduction}
	
	Graphons provide a fundamental framework for modeling dense random graphs. 
	Originally introduced as limit objects for convergent sequences of dense graphs, they offer a nonparametric representation of connection probabilities between nodes and play a central role in network statistics, graph theory, and statistical learning \citep{Lovasz2006,Diaconis2007,Lovasz2012}. 
	However, inferring a graphon from an observed graph raises fundamental difficulties, among which identifiability is one of the most prominent.

	Indeed, in the graphon model, each node is associated with an unobserved latent position
	\(U_i\in[0,1]\), and the probability that two nodes \(i\) and \(j\) are connected is
	given by \(W(U_i,U_j)\), where \(W\) denotes the graphon. This latent parametrization is
	not unique, since composing \(W\) with a measure-preserving transformation of \([0,1]\)
	yields an equivalent graphon that defines the same probability distribution on graphs.
	Thus, graphons are naturally identifiable only up to measure-preserving equivalence. In many graphon estimation procedures, it is useful to select a canonical
	representative within this equivalence class. A standard way to do so is to construct a
	canonical representation of the latent positions. One common approach is based on the
	degree function
	\[
	g(u)=\int_0^1 W(u,v)\,dv.
	\]
	If \(g\) is injective, and in particular if it is strictly monotone, then expected degrees
	distinguish latent positions and induce a canonical ordering. The motivation behind
	degree-based assumptions is therefore not that they are necessary for graphon
	identifiability itself, but that they provide a simple route to latent-position
	identifiability, which in turn fixes a representative of the graphon within its
	equivalence class. Moreover, this ordering can be approximated from empirical degrees, which explains the
	role of degree-based assumptions in sorting-based and neighborhood-smoothing graphon
	estimators \citep{Yang2014,Chan2014,Zhang2017,Sogan2026}. Although successful in many situations, this degree-based strategy is rather restrictive. It fails whenever distinct latent positions have the same expected degree, even if they have different connectivity profiles. This situation occurs, for example, in stochastic block models with equal block degrees and in finite-rank graphons with constant or non-injective degree functions. To the best of our knowledge, there is no general framework for extending this form of latent-position identifiability beyond graphons with injective degree functions.

	In this work, we introduce a new notion of identifiability based on
	\emph{rooted motif signatures}. These signatures can be viewed as node-level analogues
	of the homomorphism and subgraph densities that characterize graphons at the global
	level \citep{Lovasz2006,Diaconis2007,Lovasz2012}. Instead of summarizing the whole
	graphon by global motif densities, we attach to each latent position the collection of motif densities rooted at that point. The key idea is to enrich the local description of a node beyond its degree by
	incorporating higher-order rooted motif densities, such as rooted triangles,
	four-node motifs, cycles, and clique-like patterns. These quantities capture local
	connectivity patterns that are invisible to degree information alone and can therefore distinguish latent positions with identical expected degree but different connectivity profiles \citep{Milo2002,Alon2007,Benson2016}.
	
	We investigate the extent to which rooted motif signatures determine latent connectivity profiles. In particular, we show that, for generic finite-rank graphons, suitable rooted
	motif signatures are rich enough to identify these profiles. We also explain why such identifiability cannot hold for arbitrary graphons without additional assumptions, since certain symmetries of the latent space preserve all rooted
	motif signatures. 
	Finally, we develop the statistical counterpart of this construction. Given a single
	observed graph, we define empirical rooted motif signatures and prove uniform
	concentration bounds for these estimators. These bounds provide a statistical guarantee
	for comparing nodes through their motif signatures. Simulation experiments further
	illustrate that rooted motif signatures reveal latent structure beyond degree
	information, successfully separating nodes that remain indistinguishable under
	degree-based representations.
	%Finally, we develop the statistical counterpart of this construction. Given a single observed graph, we define empirical rooted motif signatures and prove uniform concentration bounds for these estimators. These results imply the consistency of the associated motif distances and provide a theoretical foundation for motif-based graphon estimation. We further conduct simulation experiments showing that rooted motif signatures capture latent structure beyond degree information and can differentiate nodes that remain indistinguishable under degree-based representations.

	%Finally, we develop the statistical counterpart of this construction. Given a singleobserved graph, we define empirical rooted motif signatures and prove uniform concentration bounds for these estimators. As a consequence, the induced motif distances	are estimated consistently. Simulation experiments further illustrate that rooted motif	signatures reveal latent structure beyond degree information, successfully separating nodes that remain indistinguishable under degree-based representations.

	\section{Preliminaries}
	
	\subsection{Graphons}
	
	A graphon is a symmetric measurable function
	\(
	W:[0,1]^2 \to [0,1],
	\)
	which can be interpreted as an edge-probability function for dense random graphs. To
	generate a random graph \(G=(V,E)\) on \(n\) nodes, first sample latent variables
	\(U_1,\ldots,U_n\) independently and uniformly from \([0,1]\). Conditionally on these
	latent variables, draw the entries of the adjacency matrix independently as
	\[
	A_{ij}|U_i,U_j\sim \mathrm{Bernoulli}\bigl(W(U_i,U_j)\bigr),
	\qquad 1\le i<j\le n,
	\]
	set \(A_{ji}=A_{ij}\), and put \(A_{ii}=0\). The edge set is then given by
	\(
	E=\bigl\{\{i,j\}:A_{ij}=1\bigr\}.
	\)
	In this representation, the latent variable \(U_i\) plays the role of an unobserved node
	position, while the graphon \(W\) describes how connection probabilities vary across the
	latent space. Many classical random graph models can be expressed in this form. In
	particular, stochastic block models correspond to piecewise-constant graphons. Thus,
	graphons provide a natural nonparametric extension of block models and a flexible
	framework for statistical network analysis \citep{Diaconis2007, Lovasz2012}.
	
	\subsection{Identifiability of graphons}
	
	The identifiability problem in graphon models arises from the fact that the latent space is only defined up to measure-preserving transformations \citep{Aldous1981,Hoover1979,Kallenberg2005,Diaconis2007}. Let
	\(\varphi:[0,1]\to[0,1]\) be a measure-preserving transformation, that is,
	\[
	\lambda(\varphi^{-1}(A))=\lambda(A)
	\qquad
	\text{for every measurable set } A\subset[0,1],
	\]
	where \(\lambda\) denotes the Lebesgue measure. If \(U\sim\mathrm{Unif}[0,1]\), then
	\(\varphi(U)\sim\mathrm{Unif}[0,1]\). Consequently, the transformed graphon
	\[
	W^\varphi(u,v)
	=
	W(\varphi(u),\varphi(v))
	\]
	induces the same random graph model as \(W\). This shows that graphons can be identified at most up to
	measure-preserving transformations. However, transformations of this form
	do not give the full  equivalence relation between graphons. Indeed, two graphons may induce the same random graph distribution without being related by a single
	measure-preserving transformation.
	%At first sight, one might therefore identify graphons up to transformations of this form. However, this equivalence is still too restrictive. The converse implication does not hold in general: two graphons may induce the same random graphdistribution without being related by a single measure-preserving transformation.
	\cite{Diaconis2007} give the example
	\[
	W(u,v)=uv,
	\qquad
	W'(u,v)=(2u\bmod 1)(2v\bmod 1).
	\]
	These two graphons define the same exchangeable random graph model, but there is no
	measure-preserving transformation \(\psi\) such that
	\(
	W(u,v)=W'(\psi(u),\psi(v))
	\)
	almost everywhere.
	
	The appropriate statistical notion is therefore weak equivalence: two graphons are
	weakly equivalent when they induce the same random graph distribution. 
	We write this equivalently as
	\[
	\delta_\square(W,W')=0,
	\]
	where \(\delta_\square\) denotes the cut distance \citep{Lovasz2006,Borgs2008,Diaconis2007}. Recall that, for an integrable
	function \(U\) on \([0,1]^2\), the cut norm is defined by
	\[
	\|U\|_\square
	=
	\sup_{S,T\subseteq[0,1]}
	\left|
	\int_{S\times T} U(u,v)\,du\,dv
	\right|,
	\]
	where the supremum is taken over measurable sets \(S,T\). The cut distance between two
	graphons is then
	\[
	\delta_\square(W,W')
	=
	\inf_{\varphi,\varphi'}
	\left\|
	W^\varphi-(W')^{\varphi'}
	\right\|_\square,
	\]
	where the infimum is taken over measure-preserving transformations
	\(\varphi,\varphi':[0,1]\to[0,1]\). The following result
	gives a characterization of this equivalence.
	
	\begin{theorem}[\cite{Diaconis2007}, Theorem 7.1]
		Let \(W\) and \(W'\) be two graphons. Then
		\(
		\delta_\square(W,W')=0
		\)
		if and only if there exist measure-preserving transformations
		\(
		\varphi,\varphi':[0,1]\to[0,1]
		\)
		such that
		\[
		W(\varphi(u),\varphi(v))
		=
		W'(\varphi'(u),\varphi'(v)), \qquad \text{for almost every } (u,v)\in[0,1]^2.
		\]

	\end{theorem}
	
	Thus, the observed graph distribution identifies only the weak equivalence class of
	\(W\), not a unique representative of this class. A canonical representative can be
	obtained by following the approach of \cite{Borgs2010}, which studies identifiability at the level of latent positions through their connectivity profiles. Two latent positions
	\(u,v\in[0,1]\) are called twins if
	\(
	W(u,\cdot)=W(v,\cdot)\;
	\text{almost everywhere}.
	\)
	Twin latent positions have identical connectivity profiles. As a result, any statistic that depends only on the connection function \(W(u,\cdot)\) assigns the same value to such positions. The twin relation therefore defines a quotient of the latent space. Writing
	\[
	u\sim_{\mathrm{twin}} v
	\quad\Longleftrightarrow\quad
	W(u,\cdot)=W(v,\cdot)\ \text{a.e.},
	\]
	one obtains the quotient space
	\(
	\Omega_W=[0,1]/\!\sim_{\mathrm{twin}}.
	\)
	The elements of \(\Omega_W\) are not necessarily intervals of \([0,1]\). They are
	equivalence classes of latent positions having the same connectivity profile. In a
	stochastic block model, these classes correspond to the blocks, provided that distinct blocks have distinct connectivity profiles; blocks with identical rows in the block
	probability matrix are identified in the same quotient class. The Lebesgue measure induces a probability measure \(\mu_W\) on \(\Omega_W\), and the
	quotient graphon is defined almost everywhere by
	\(
	\widetilde W([u],[v])=W(u,v),
	\). Equivalently, if
	\(q_W:[0,1]\to\Omega_W\), \(q_W(u)=[u]\) denotes the quotient map, then
	\[
	W(u,v)=\widetilde W(q_W(u),q_W(v))
	\quad\text{for almost every }(u,v).
	\]
	Thus, \(\widetilde W\) induces the same random graph distribution as \(W\). It is
	therefore a representative of the same weak equivalence class, but defined on the
	quotient probability space \((\Omega_W,\mu_W)\) rather than on the original interval.
	Results of \cite{Borgs2010} show that, within the class of twin-free representatives,
	weak equivalence coincides with isomorphism up to null sets. Consequently, twin-freeness
	removes the ambiguity created by duplicated latent positions; it is therefore a necessary
	condition for any notion of latent-position identifiability based on connectivity
	profiles.
	
	The quotient construction removes duplicated connectivity profiles, but it does not
	provide a scalar parametrization of the latent space. A classical way to obtain such a
	parametrization is to restrict attention to graphons whose latent positions can be
	distinguished by their expected degrees. This leads to the degree function
	\[
	g(u)
	=
	\int_0^1 W(u,v)\,dv
	=
	\mathbb E\!\left[A_{ij}\mid U_i=u\right],
	\]
	which is the first conditional moment of the connectivity profile of a node with latent
	position \(u\). In particular, if the graphon admits a representative with a strictly
	monotone degree function, then expected degrees induce a canonical ordering of the latent
	space.
	
	\begin{condition}[Strict monotonicity of degree]
		\label{cond:degree-monotonicity}
		A graphon \(W\) is said to satisfy the strict degree monotonicity condition if and only if
		there exists a representative \(W^{\mathrm{can}}\) in the weak equivalence class of \(W\)
		such that
		\[
		g^{\mathrm{can}}(u)
		=
		\int_0^1 W^{\mathrm{can}}(u,v)\,dv
		\]
		is strictly increasing (or strictly decreasing) on \([0,1]\).
	\end{condition}
	Under this condition, latent positions are identifiable through their expected degrees, and the degree function provides a one-dimensional coordinate system on the latent space.
	This idea has been used in several graphon estimation methods
	\citep{Chan2014,Yang2014,Olhede2014,Wolfe2013,Sogan2026}
	and is explicitly presented as a way to define a canonical graphon representation
	\citep{Yang2014,Chan2014}.
	
	However, strict monotonicity of degree is a restrictive assumption. It excludes graphons
	in which distinct latent positions have the same expected degree but different
	connectivity profiles. In such cases, the degree function cannot separate all latent
	positions, even though the corresponding graphon may still be twin-free. This limitation
	motivates replacing the scalar degree coordinate by richer local summaries of the
	connectivity profile.

	\section{Rooted local signatures and motif-based indistinguishability}
	
	The expected degree only captures first-order connectivity information. To obtain a
	richer description of a latent position \(u\), we consider conditional subgraph
	densities rooted at \(u\). These quantities generalize the degree function by measuring
	the expected occurrence of local motifs around a node located at \(u\). They can be
	viewed as higher-order conditional moments of the connectivity profile
	\(W(u,\cdot)\).
	
	\begin{definition}[Rooted motif density]
		Let \((F,r)\) be a rooted graph, where \(F=(V(F),E(F))\) is a finite simple graph and
		\(r\in V(F)\) is a distinguished node called the root. For \(u\in[0,1]\), the rooted motif density of \((F,r)\) at \(u\) is
		\[
		t((F,r),W)(u)
		=
		\int_{[0,1]^{|V(F)|-1}}
		\prod_{\{a,b\}\in E(F)}
		W(u_a,u_b)
		\prod_{v\in V(F)\setminus\{r\}}du_v,
		\]
		where \(u_r=u\).
	\end{definition}
	The quantity \(t((F,r),W)(u)\) has a natural probabilistic interpretation. Fix the root node at latent position \(u\), sample the remaining \(|V(F)|-1\) latent positions independently from the uniform distribution on \([0,1]\), and generate edges according to the graphon \(W\). Then \(t((F,r),W)(u)\) is the probability that all edges of \(F\) are present among these nodes, with the root node fixed at \(u\). It is therefore a conditional subgraph density attached to the latent
	position \(u\). The degree function is recovered as the simplest rooted motif density. If
	\(K_2^\bullet\) denotes the rooted edge, then
	\[
	t(K_2^\bullet,W)(u)
	=
	\int_0^1 W(u,v)\,dv
	=
	g(u).
	\]
	Similarly, if \(P_3^\bullet\) denotes the path of length two rooted at one extremity,
	then
	\[
	t(P_3^\bullet,W)(u)
	=
	\int_{[0,1]^2}
	W(u,v)W(v,w)\,dv\,dw.
	\]
	This quantity measures the expected number of two-step connections emanating from a node
	located at \(u\).
	
	\subsection{Rooted motif signatures}
	
	The rooted motif densities introduced above can be collected into a single descriptor of
	the latent position \(u\). Since they summarize the occurrence of all finite rooted
	motifs around \(u\), they provide a much richer characterization than the expected
	degree alone.
	
	Let \(\mathcal F_\infty\) denote the collection of all finite rooted motifs, considered up to
	rooted isomorphism.
	
	\begin{definition}[Rooted motif signature]
		The rooted motif signature of a latent position \(u\in[0,1]\) is the collection
		\[
		\Phi_W(u)
		=
		\bigl(
		t((F,r),W)(u)
		\bigr)_{(F,r)\in\mathcal F_\infty}.
		\]
	\end{definition}
	
	The signature \(\Phi_W(u)\) records all rooted motif densities around \(u\). In
	particular, the degree function appears as the rooted edge
	motif,
	\(
	t(K_2^\bullet,W)(u)
	=
	g(u).
	\)
	The signature therefore extends degree-based representations by incorporating
	higher-order local structure. Two latent positions having the same signature are indistinguishable through rooted motif
	statistics. This naturally induces a partition of the latent space into classes of motif-indistinguishable positions.
	In practice, one cannot estimate infinitely many motif densities. It is therefore
	natural to restrict attention to a finite collection of motifs. Let
	\(
	\mathcal F
	=
	\{(F_1,r_1),\ldots,(F_K,r_K)\}
	\)
	be a finite family of rooted motifs. The associated \(\mathcal F\)-signature is
	\[
	\Phi_W^{\mathcal F}(u)
	=
	\Bigl(
	t((F_1,r_1),W)(u),
	\ldots,
	t((F_K,r_K),W)(u)
	\Bigr).
	\]
	For instance, choosing
	\[
	\mathcal F
	=
	\{K_2^\bullet,\; K_3^\bullet,\; P_3^\bullet,\; C_4^\bullet\},
	\]
	where \(K_2^\bullet\) is the rooted edge, \(K_3^\bullet\) is the rooted triangle,
	\(P_3^\bullet\) is the rooted path of length two, and \(C_4^\bullet\) is the rooted
	four-cycle. The corresponding signature contains
	\[
	\Phi_W^{\mathcal F}(u)
	=
	\bigl(
	g(u),
	g_{\triangle}(u),
	g_{P_3}(u),
	g_{C_4}(u)
	\bigr),
	\]
	with, for example,
	\[
	g_{\triangle}(u)
	=
	\int_{[0,1]^2}
	W(u,v)W(u,w)W(v,w)\,dv\,dw.
	\]
	\begin{definition}[Signature-based latent-position identifiability]
		A graphon \(W\) is said to be \emph{signature-identifiable} if the rooted motif
		signature map
		\(
		u \mapsto \Phi_W(u)
		\)
		is injective almost everywhere.
	\end{definition}
	This definition is a latent-position identifiability condition within a given
	representative of the graphon. It says that two latent positions cannot have the same rooted motif signature, except possibly on a set of measure zero.
	Similarly, for a finite family \(\mathcal F\) of rooted motifs, we say that \(W\) is \(\mathcal F\)-signature identifiable if the finite-dimensional map
	\(
	u\mapsto \Phi_W^{\mathcal F}(u)
	\) is injective almost everywhere. The degree function is recovered as the simplest rooted
	motif coordinate. Indeed, when
	\(
	\mathcal F=\{K_2^\bullet\},
	\)
	\(\mathcal F\)-signature identifiability coincides with almost-everywhere injectivity of the degree function. Thus, degree-based identifiability appears as the one-dimensional case of the rooted motif framework. More generally, if the degree function is injective and \(K_2^\bullet\in\mathcal F\), then the signature map \(u\mapsto\Phi_W^{\mathcal F}(u)\) is injective. The converse does not necessarily hold: rooted motif signatures may distinguish latent positions that have the same expected degree. The previous definition only requires separation of latent positions. It does not require the signature coordinates to induce an ordered representative. To formulate an analogue
	of strict degree monotonicity, one may impose the following stronger condition.
	
	\begin{definition}[Strictly ordered signature representation]
		Let \(\mathcal F\) be a finite family of rooted motifs. We say that a graphon \(W\)
		admits a strictly \(\mathcal F\)-signature ordered representation if there exists a
		representative \(W^{\mathrm{can}}\) in the weak equivalence class of \(W\) such that,
		for almost every \(u<v\),
		\[
		t((F,r),W^{\mathrm{can}})(u)
		\le
		t((F,r),W^{\mathrm{can}})(v)
		\qquad
		\text{for all } (F,r)\in\mathcal F,
		\]
		and there exists at least one motif \((F,r)\in\mathcal F\) such that
		\[
		t((F,r),W^{\mathrm{can}})(u)
		<
		t((F,r),W^{\mathrm{can}})(v).
		\]
	\end{definition}
	This condition is stronger than \(\mathcal F\)-signature identifiability. Signature
	identifiability only requires the map
	\[
	u\mapsto \Phi_W^{\mathcal F}(u)
	\]
	to separate latent positions almost everywhere. In contrast, the strictly ordered
	condition requires the existence of a representative in which all motif coordinates are
	nondecreasing along the latent order, with at least one coordinate strictly increasing
	for almost every pair \(u<v\). Hence, not every \(\mathcal F\)-signature identifiable
	graphon admits such a representation. The monotone representation should therefore be
	viewed as one possible way to select an ordered representative, analogous to strict
	degree monotonicity. Other choices are possible, for instance by ordering latent
	positions lexicographically according to their motif signatures.
	When \(\mathcal F=\{K_2^\bullet\}\), this condition reduces to strict monotonicity of
	the degree function.

	\subsection{Finite-rank graphons}
	
	The previous definitions apply to arbitrary graphons. In general, however, it is not clear whether equality of rooted motif signatures implies equality of connectivity profiles. We therefore consider a class of graphons for which the latent structure can be analyzed explicitly. We now specialize to finite-rank graphons. Such graphons admit finite spectral representations through the associated graphon operator and have been studied in several contexts including graphon estimation, graphon centrality, and graphon control	\citep{Wolfe2013,AvellaMedina2018,Gao2019}.
	A graphon \(W\) is said to have finite rank \(m\) if the associated integral operator
	\begin{equation}
		(T_Wf)(u)
		=
		\int_0^1 W(u,v)f(v)\,dv
	\end{equation}
	has rank \(m\). Equivalently, \(W\) admits a spectral decomposition
	\begin{equation}
		W(u,v)
		=
		\sum_{j=1}^{m}
		\lambda_j \phi_j(u)\phi_j(v),
	\end{equation}
	where \(\lambda_1,\ldots,\lambda_m\) are the nonzero eigenvalues of \(T_W\), and
	\(\phi_1,\ldots,\phi_m\) form an orthonormal family in \(L^2([0,1])\).
	For finite-rank graphons, the connectivity profile of a latent position \(u\) is completely determined by the spectral coordinates
	\(
	\Theta(u)
	=
	\bigl(
	\phi_1(u),\ldots,\phi_m(u)
	\bigr).
	\)
	Indeed,
	\(
	W(u,\cdot)
	=
	\sum_{j=1}^{m}
	\lambda_j \phi_j(u)\phi_j(\cdot),
	\)
	so that
	\(
	W(u,\cdot)=W(v,\cdot)
	\)
	if and only if
	\(
	\Theta(u)=\Theta(v).
	\)
	\begin{proposition}\label{prop:finite-rank-path}
		Let \(W\) be a finite-rank graphon of rank \(m\). If
		\(
		\lambda_i\neq\lambda_j,
		\, (i\neq j),
		\)
		and
		\(
		\langle 1,\phi_j\rangle\neq 0\), for \(j=1,\ldots,m,
		\) then the rooted path densities
		\begin{equation}
			t(P_k^\bullet,W)(u),
			\qquad
			k=2,\ldots,m+1,
		\end{equation}
		determine the spectral coordinates \(\Theta(u)\).
	\end{proposition}
	\begin{remark}
		Proposition~\ref{prop:finite-rank-path} gives an explicit sufficient family of rooted
		motifs for generic finite-rank graphons. It shows that the rooted path densities
		\(t(P_k^\bullet,W)(u)\), \(k=2,\ldots,m+1\), determine the spectral coordinates
		\(\Theta(u)\), and hence the connectivity profile \(W(u,\cdot)\). Therefore, any rooted
		motif signature containing these path coordinates identifies twin classes in this
		setting. This result should not be interpreted as saying that rooted paths are the only
		informative motifs. They are used here because their relation to the spectral
		decomposition is explicit. Other rooted motifs may also distinguish latent positions, but
		the proposition only establishes the sufficiency of rooted paths.
	\end{remark}
	A natural question is whether an analogous identifiability phenomenon can hold beyond finite-rank graphons.
	More precisely, one may ask whether
	\begin{equation}
		\Phi_W(u)=\Phi_W(v)
		\qquad\Longrightarrow\qquad
		W(u,\cdot)=W(v,\cdot)
		\quad\text{a.e.}
	\end{equation}
	For arbitrary graphons, this implication is false. The obstruction comes from internal symmetries of the graphon. Suppose that there exists a measure-preserving transformation \(\tau:[0,1]\to[0,1]\) such that
	\(
	W(\tau(x),\tau(y))=W(x,y)
	\,\text{for a.e. }(x,y)\in[0,1]^2.
	\)
	Then, by applying the change of variables \(z\mapsto \tau(z)\) in the definition of each
	rooted motif density, one obtains
	\(
	t((F,r),W)(\tau(u))
	=
	t((F,r),W)(u)
	\)
	for every rooted motif \((F,r)\) and for almost every \(u\). Hence
	\(
	\Phi_W(\tau(u))=\Phi_W(u)
	\,\text{for a.e. }u.
	\)
	This phenomenon occurs even for twin-free graphons. For instance, consider
	\(
	W(u,v)=|u-v|.
	\)
	The map
	\(
	\tau(u)=1-u
	\)
	is measure-preserving and satisfies
	\(
	W(\tau(u),\tau(v))=W(u,v)
	\,\text{for all }u,v\in[0,1].
	\)
	Therefore,
	\[
	\Phi_W(u)=\Phi_W(1-u)
	\qquad\text{for all }u\in[0,1].
	\]
	However, if
	\(
	|u-\cdot|=|v-\cdot|
	\,\text{a.e.},
	\)
	then necessarily \(u=v\), so \(W\) is twin-free. Thus twin-freeness alone does not imply signature-based latent-position identifiability.
	This example shows that the full rooted motif signature cannot, in general, identify latent positions beyond the symmetries of the graphon. This suggests the following natural question: if
	\(
	\Phi_W(u)=\Phi_W(v),
	\) must \(u\) and \(v\) be related by a measure-preserving transformation \(\tau\) satisfying
	\(
	W(\tau(x),\tau(y))=W(x,y)
	\,\text{for a.e. }(x,y),
	\)
	with
	\(
	\tau(u)=v?
	\)
	Equivalently, do rooted motif signatures identify latent positions up to the internal
	measure-preserving symmetries of \(W\)? We leave this question open.

	\section{Empirical rooted motif signatures}

	Degree-based graphon estimators use empirical degrees to recover a latent ordering under
	strict degree monotonicity \citep{Yang2014,Chan2014}. We extend this idea by replacing
	the scalar degree coordinate with a finite vector of empirical rooted motif densities, which can distinguish latent positions that have identical expected degree. 
	
	Let
	\(
	\mathcal F
	=
	\{(F_1,r_1),\ldots,(F_K,r_K)\}
	\)
	be a finite family of rooted motifs.  Suppose that we observe the adjacency matrix
	\(
	A=(A_{ij})_{1\leq i,j\leq n}
	\)
	of an undirected graph generated from the graphon \(W\). Throughout this section, the motif family \(\mathcal F\) is fixed, and we assume that
	\(n\) is large enough so that
	\(
	|V(F)|\leq n,
	\,
	(F,r)\in\mathcal F.
	\) For each node \(i\), define the
	empirical \(\mathcal F\)-signature by
	\begin{equation}
		\widehat{\Phi}^{\mathcal F}_i
		=
		\Bigl(
		\widehat t_i(F_1),
		\ldots,
		\widehat t_i(F_K)
		\Bigr),
	\end{equation}
	where \(\widehat t_i(F)\) denotes the empirical rooted motif density of \(F\) centered at
	node \(i\). More precisely, if \((F,r)\) is a rooted motif, then
	\begin{equation}
		\widehat t_i(F)
		=
		\frac{1}{(n-1)_{|V(F)|-1}}
		\sum_{\substack{\psi:V(F)\setminus\{r\}\hookrightarrow [n]\setminus\{i\}}}
		\prod_{\{a,b\}\in E(F)}
		A_{\psi_i(a)\psi_i(b)},
	\end{equation}
	where \((n-1)_k=(n-1)(n-2)\cdots(n-k)\), the sum is over injective maps, and
	\(
	\psi_i(r)=i\),
	\(
	\psi_i(a)=\psi(a)\), for \(a\neq r\).
	For example, the empirical rooted edge, triangle, path, and four-cycle densities are
	given by
	\begin{align}
		\widehat g_i
		&=
		\frac{1}{n-1}
		\sum_{j\neq i}
		A_{ij},\\
		\widehat g_{\triangle,i}
		&=
		\frac{1}{(n-1)(n-2)}
		\sum_{\substack{j,k\neq i\\ j\neq k}}
		A_{ij}A_{ik}A_{jk},\\
		\widehat g_{P_3,i}
		&=
		\frac{1}{(n-1)(n-2)}
		\sum_{\substack{j,k\neq i\\ j\neq k}}
		A_{ij}A_{jk},\\
		\widehat g_{C_4,i}
		&=
		\frac{1}{(n-1)(n-2)(n-3)}
		\sum_{\substack{j,k,\ell\neq i\\ j,k,\ell\ \mathrm{distinct}}}
		A_{ij}A_{jk}A_{k\ell}A_{\ell i}.
	\end{align}
	
	Given nonnegative weights
	\(
	w_1,\ldots,w_K
	\),
	the empirical signatures induce the node-level dissimilarity
	\begin{align}
		\widehat d_{\mathcal F,w}(i,j)
		&=
		\left(
		\sum_{k=1}^K
		w_k
		\left[
		\widehat t_i(F_k)
		-
		\widehat t_j(F_k)
		\right]^2
		\right)^{1/2}.
	\end{align}
	At the population level, the corresponding pseudometric is
	\begin{align}
		d_{\mathcal F,w}(u,v)
		&=
		\left(
		\sum_{k=1}^K
		w_k
		\left[
		t((F_k,r_k),W)(u)
		-
		t((F_k,r_k),W)(v)
		\right]^2
		\right)^{1/2}.
	\end{align}
	The weights allow different motif coordinates to contribute at different scales. The
	unweighted Euclidean distance corresponds to \(w_k=1\) for all \(k\). In practice, the
	weights may be chosen to standardize the motif coordinates, for instance using inverse
	empirical variances, or to penalize larger motifs. A simple size-based choice is
	\(
	w_k=|E(F_k)|^{-\alpha},
	\,\alpha\in[0,1].
	\)
	When \(\alpha\) is close to \(0\), all motifs contribute nearly equally. When
	\(\alpha\) is close to \(1\), larger motifs receive smaller weights, which reduces the
	influence of more complex and typically more variable motif counts. By construction,
	\[
	d_{\mathcal F}(u,v)=0
	\quad\Longleftrightarrow\quad
	\Phi_W^{\mathcal F}(u)=\Phi_W^{\mathcal F}(v).
	\]
	Hence \(d_{\mathcal F}\) becomes a metric on the quotient space induced by \(\mathcal F\)-motif indistinguishability. Under \(\mathcal F\)-signature identifiability, it defines a metric on the latent space up to null sets. This metric endows the latent space with a geometry induced by rooted motif statistics. Nodes that are close in this geometry have similar motif profiles and therefore similar
	local structural roles in the network. Such distances can be used to construct motif-based neighborhoods, perform clustering,
	learn low-dimensional representations of the latent space, and develop graphon estimation procedures.

	The following result shows that empirical rooted motif signatures consistently estimate
	their population counterparts.
	
	\begin{theorem}[Uniform concentration of empirical rooted signatures]
		\label{thm:empirical-signature-concentration}
		Let
		\[
		\mathcal F=\{(F_1,r_1),\ldots,(F_K,r_K)\}
		\]
		be a fixed finite family of rooted motifs, and set
		\(
		s_{\max}=\max_{1\leq k\leq K}|V(F_k)|.
		\)
		Assume that \(n\geq s_{\max}\). Then there exists a constant \(C>0\), depending only on
		\(\mathcal F\), such that for every \(\delta\in(0,1)\), with probability at least
		\(1-\delta\),
		\begin{equation}
			\max_{1\leq i\leq n}
			\left\|
			\widehat\Phi_i^{\mathcal F}
			-
			\Phi_W^{\mathcal F}(U_i)
			\right\|_2
			\leq
			C
			\sqrt{
				\frac{\log(nK)+\log(1/\delta)}
				{n}
			}.
		\end{equation}
		In particular, if \(K\) is fixed, then
		\begin{equation}
			\max_{1\leq i\leq n}
			\left\|
			\widehat\Phi_i^{\mathcal F}
			-
			\Phi_W^{\mathcal F}(U_i)
			\right\|_2
			=
			O_{\mathbb P}
			\left(
			\sqrt{\frac{\log n}{n}}
			\right).
		\end{equation}
	\end{theorem}
	Theorem~\ref{thm:empirical-signature-concentration} shows that, uniformly over the
	observed nodes, empirical rooted motif signatures recover their population counterparts.
	Thus, for a fixed motif family \(\mathcal F\), the vectors
	\(\widehat\Phi_i^{\mathcal F}\) provide consistent data-driven coordinates for the
	latent motif representation. Since many downstream procedures depend on pairwise comparisons rather than on the
	individual signatures themselves, we next derive a uniform consistency result for the
	induced motif distances.
	
	\begin{corollary}[Uniform consistency of signature distances]
		\label{cor:signature-distance}
		Under the assumptions of Theorem~\ref{thm:empirical-signature-concentration},
		there exists a constant \(C>0\), depending only on \(\mathcal F\), such that for every \(\delta\in(0,1)\),
		with probability at least \(1-\delta\),
		\[
		\max_{1\le i,j\le n}
		\left|
		\widehat d_{\mathcal F,w}(i,j)
		-
		d_{\mathcal F,w}(U_i,U_j)
		\right|
		\le
		C
		\left(
		\sum_{k=1}^K w_k
		\right)^{1/2}
		\sqrt{
			\frac{\log(nK)+\log(1/\delta)}
			{n}
		}.
		\]
		In particular, if
		\(
		\sum_{k=1}^K w_k =1,
		\)
		then
		\[
		\max_{1\le i,j\le n}
		\left|
		\widehat d_{\mathcal F,w}(i,j)
		-
		d_{\mathcal F,w}(U_i,U_j)
		\right|
		=
		O_{\mathbb P}
		\!\left(
		\sqrt{\frac{\log(nK)}{n}}
		\right).
		\]
	\end{corollary}

	\begin{figure}[htbp]
		\centering
		
		\begin{subfigure}{0.23\textwidth}
			\centering
			\includegraphics[width=\linewidth]{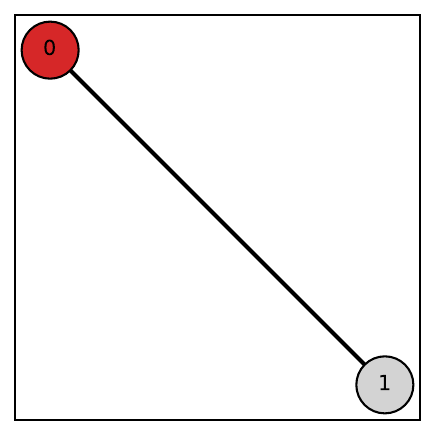}
			\caption{\(K_2^\bullet\)}
		\end{subfigure}
		\hfill
		\begin{subfigure}{0.23\textwidth}
			\centering
			\includegraphics[width=\linewidth]{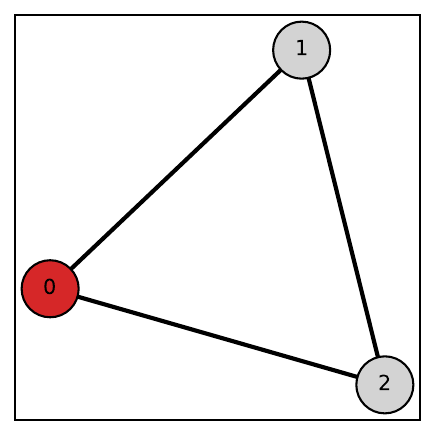}
			\caption{\(K_3^\bullet\)}
		\end{subfigure}
		\hfill
		\begin{subfigure}{0.23\textwidth}
			\centering
			\includegraphics[width=\linewidth]{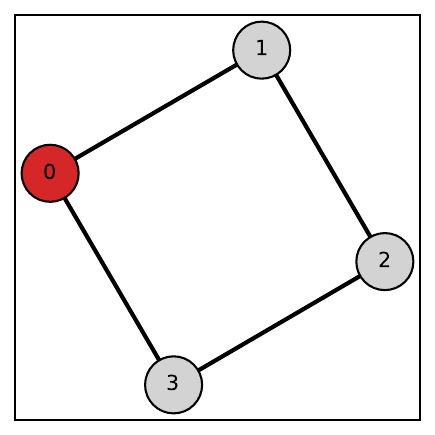}
			\caption{\(C_4^\bullet\)}
		\end{subfigure}
		\hfill
		\begin{subfigure}{0.23\textwidth}
			\centering
			\includegraphics[width=\linewidth]{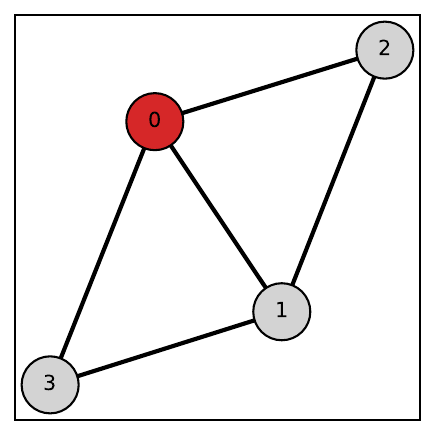}
			\caption{\(D^\bullet\)}
		\end{subfigure}
		
		\vspace{0.4cm}
		
		\begin{subfigure}{0.23\textwidth}
			\centering
			\includegraphics[width=\linewidth]{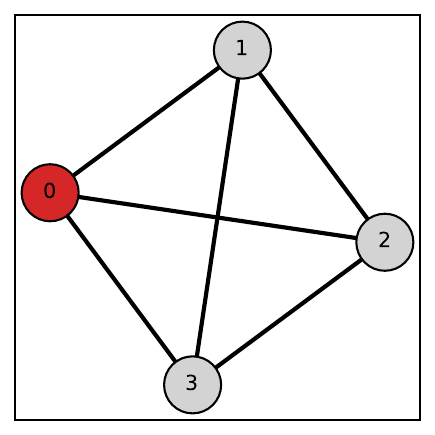}
			\caption{\(K_4^\bullet\)}
		\end{subfigure}
		\hfill
		\begin{subfigure}{0.23\textwidth}
			\centering
			\includegraphics[width=\linewidth]{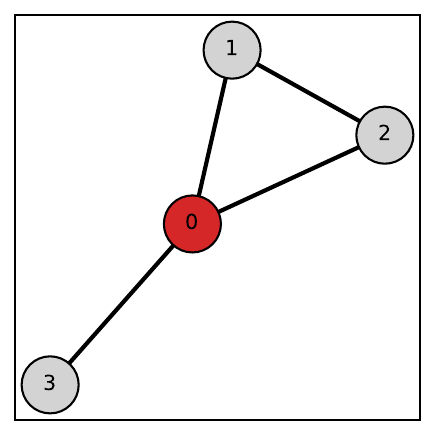}
			\caption{\(\mathrm{Paw}^\bullet\)}
		\end{subfigure}
		\hfill
		\begin{subfigure}{0.23\textwidth}
			\centering
			\includegraphics[width=\linewidth]{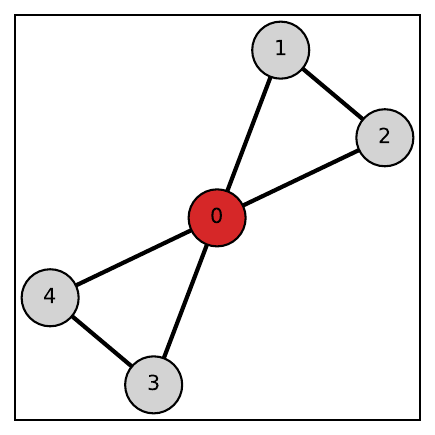}
			\caption{\(\mathrm{Butterfly}^\bullet\)}
		\end{subfigure}
		\hfill
		\begin{subfigure}{0.23\textwidth}
			\centering
			\includegraphics[width=\linewidth]{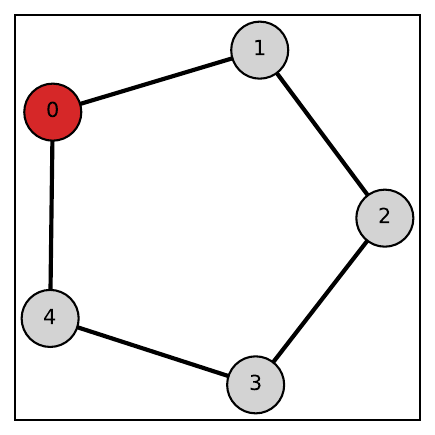}
			\caption{\(C_5^\bullet\)}
		\end{subfigure}
		
		\vspace{0.4cm}
		
		\begin{subfigure}{0.23\textwidth}
			\centering
			\includegraphics[width=\linewidth]{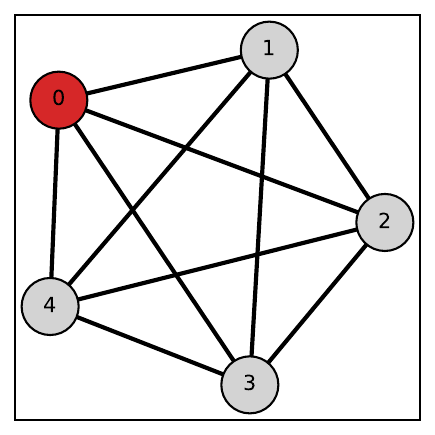}
			\caption{\(K_5^\bullet\)}
		\end{subfigure}
		\hfill
		\begin{subfigure}{0.23\textwidth}
			\centering
			\includegraphics[width=\linewidth]{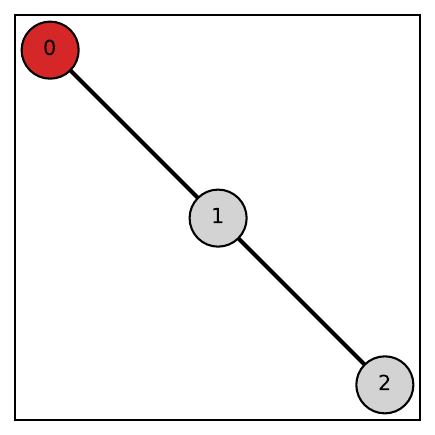}
			\caption{\(P_2^\bullet\)}
		\end{subfigure}
		\hfill
		\begin{subfigure}{0.23\textwidth}
			\centering
			\includegraphics[width=\linewidth]{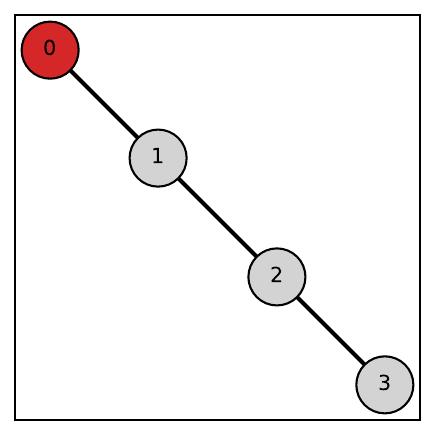}
			\caption{\(P_3^\bullet\)}
		\end{subfigure}
		\hfill
		\begin{subfigure}{0.23\textwidth}
			\centering
			\includegraphics[width=\linewidth]{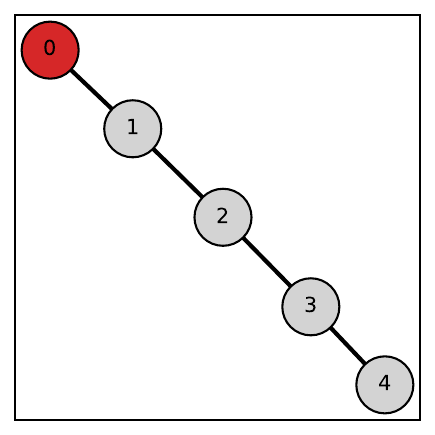}
			\caption{\(P_4^\bullet\)}
		\end{subfigure}
		
		\caption{Rooted motifs used in the numerical experiments. The root vertex is highlighted in red.}
		\label{fig:rooted-motifs-grid}
	\end{figure}

	\section{Numerical experiments}

	In the numerical experiments, we do not restrict the signature to rooted paths. Although
	rooted paths are useful theoretically, in particular for the finite-rank result above,
	they may be highly correlated and therefore need not provide the most informative
	signature in practice. We instead use a collection of small rooted motifs designed to
	capture complementary local structures, including first-order and higher-order
	connectivity patterns. The selected rooted motifs are
	\[
	\mathcal F
	=
	\left\{
	K_2^\bullet,
	K_3^\bullet,
	C_4^\bullet,
	D^\bullet,
	K_4^\bullet,
	\mathrm{Paw}^\bullet,
	\mathrm{Butterfly}^\bullet,
	C_5^\bullet,
	K_5^\bullet,
	P_2^\bullet,\ldots,P_L^\bullet
	\right\}.
	\]
	These are illustrated in Figure~\ref{fig:rooted-motifs-grid}

	\subsection{Experiment 1: Stochastic block models}\label{subsec:experiment1}
	
	We first consider a collection of stochastic block models, a classical family of latent network models widely used for community detection and graphon approximation \citep{Holland1983,Bickel2009,Abbe2018}. The models are designed to test whether rooted motif signatures can distinguish latent classes beyond degree information. Each graph is
	generated from an stochastic block model with \(Q\) blocks, block proportions \(\pi\), and connectivity matrix
	\[
	P(A_{ij}=1\mid Z_i=a,Z_j=b)= B_{ab},
	\]
	where \(B\in[0,1]^{Q\times Q}\) is symmetric and \(Z_i\in\{1,\ldots,Q\}\) denotes the
	latent block label of node \(i\). We choose six stochastic block models, summarized in Table~\ref{tab:sbm-library}. The models are chosen to include assortative, disassortative, equal-degree, role-like, cyclic, and higher-dimensional block structures, which are common benchmarks for community detection and latent block modeling \citep{Nowicki2001,Rohe2011,Lei2015}.

	For each model, we generate one graph with \(n=1000\) nodes and compute, for every node,
	its empirical rooted motif signature. The signature includes the empirical degree, rooted
	triangles, rooted four-cycles, rooted diamonds, rooted \(K_4\) cliques, rooted paw
	motifs, rooted butterfly motifs, rooted five-cycles, rooted \(K_5\) cliques, and
	walk-based rooted path coordinates. The resulting signature vectors are standardized and projected onto their first two
	principal components for visualization \citep{Jolliffe2002}. We compare this two-dimensional motif representation with the degree-only representation.
	
	Figure~\ref{fig:degree-vs-motif-pca-sbm} displays the results. In the left panel of each subfigure, nodes are represented only by their empirical degree, with a small vertical jitter added for visualization. In the right panel, nodes are represented by the first two principal components of their rooted motif signatures. Colors indicate the true blocks, and ellipses summarize the dispersion of each block. The degree-only representation often produces substantial overlap between latent blocks, especially in the equal-degree setting. In contrast, the projection of rooted motif signatures onto the first two principal components provides a clearer separation of the latent blocks. This illustrates that higher-order rooted motif information captures structural differences between blocks that are not visible from degree information alone.
	
	\begin{table}[htbp]
		\centering
		\caption{SBM models used and connectivity matrices are in the Appendix~\ref{fig:sbm-library-heatmaps}.}
		\label{tab:sbm-library}
		\begin{tabular}{lllp{6.2cm}}
			\toprule
			Model & \(Q\) & \(\pi\) & Description \\
			\midrule
			Assortative-2 & 2 & \((0.45,0.55)\)
			& Two assortative blocks with unequal within-block strengths. \\
			
			Disassortative-2 & 2 & \((0.40,0.60)\)
			& Two disassortative blocks without exchange symmetry. \\
			
			Equal-degree-3 & 3 & \((1/3,1/3,1/3)\)
			& Three balanced blocks with equal row sums but no block automorphism. \\
			
			Role-3 & 3 & \((0.20,0.45,0.35)\) & Three asymmetric role-like blocks. \\
			
			Broken-cycle-4 & 4 & \((0.18,0.27,0.25,0.30)\)
			& Four-block cycle structure with broken rotational and reflection symmetries. \\
			
			Gradient-5 & 5 & \((0.14,0.19,0.22,0.25,0.20)\)
			& Five-block gradient with unequal block masses. \\
			\bottomrule
		\end{tabular}
	\end{table}

	\begin{figure}[htbp]
		\centering
		
		\begin{subfigure}{0.48\textwidth}
			\includegraphics[width=\linewidth]{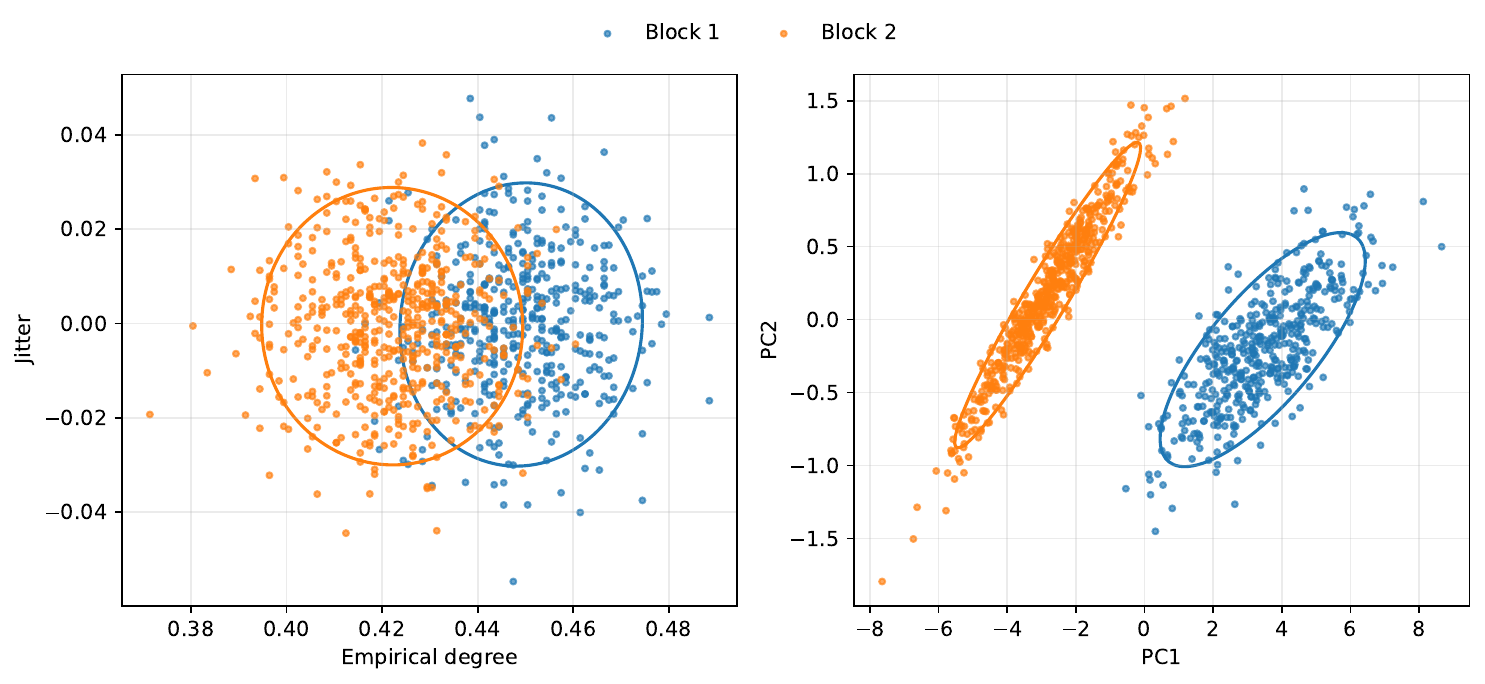}
			\caption{Assortative-2}
		\end{subfigure}
		\hfill
		\begin{subfigure}{0.48\textwidth}
			\includegraphics[width=\linewidth]{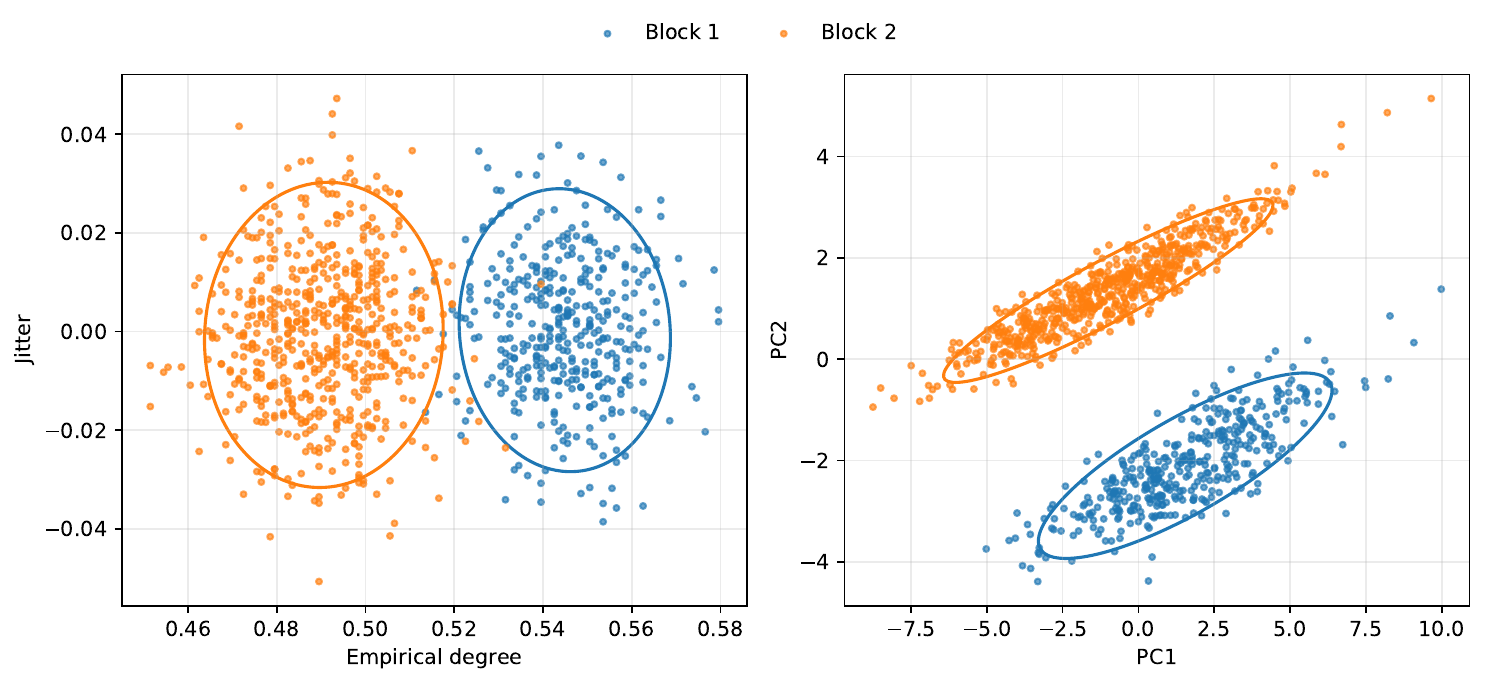}
			\caption{Disassortative-2}
		\end{subfigure}
		
		\vspace{0.3cm}
		
		\begin{subfigure}{0.48\textwidth}
			\includegraphics[width=\linewidth]{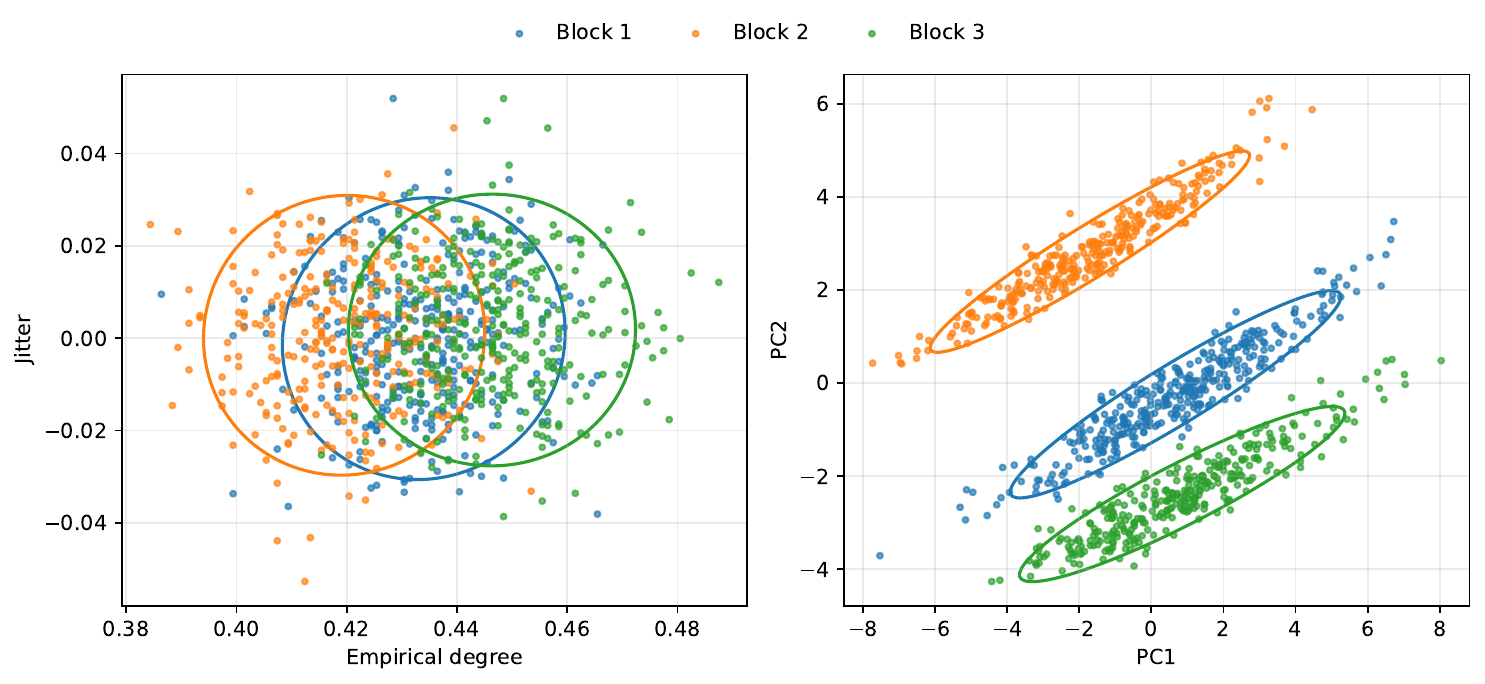}
			\caption{Equal-degree-3}
		\end{subfigure}
		\hfill
		\begin{subfigure}{0.48\textwidth}
			\includegraphics[width=\linewidth]{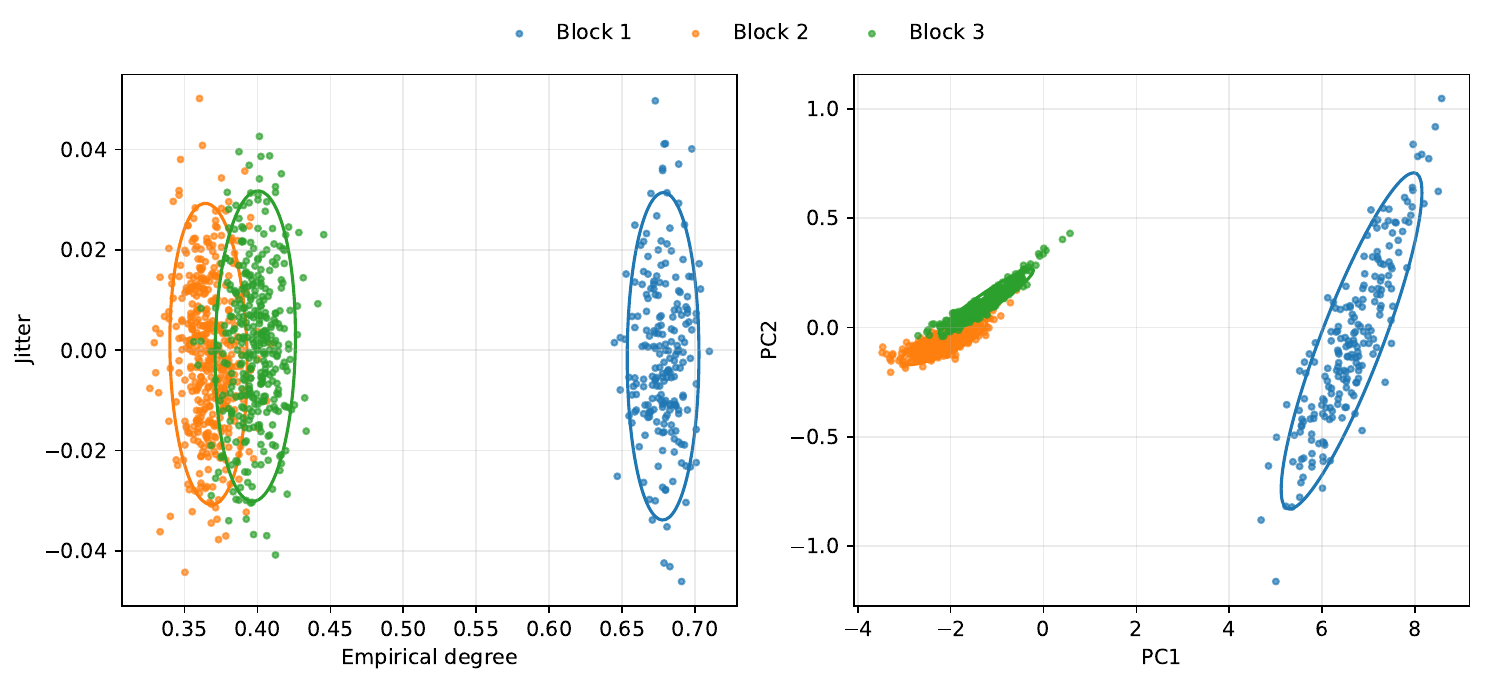}
			\caption{Role-3}
		\end{subfigure}
		
		\vspace{0.3cm}
		
		\begin{subfigure}{0.48\textwidth}
			\includegraphics[width=\linewidth]{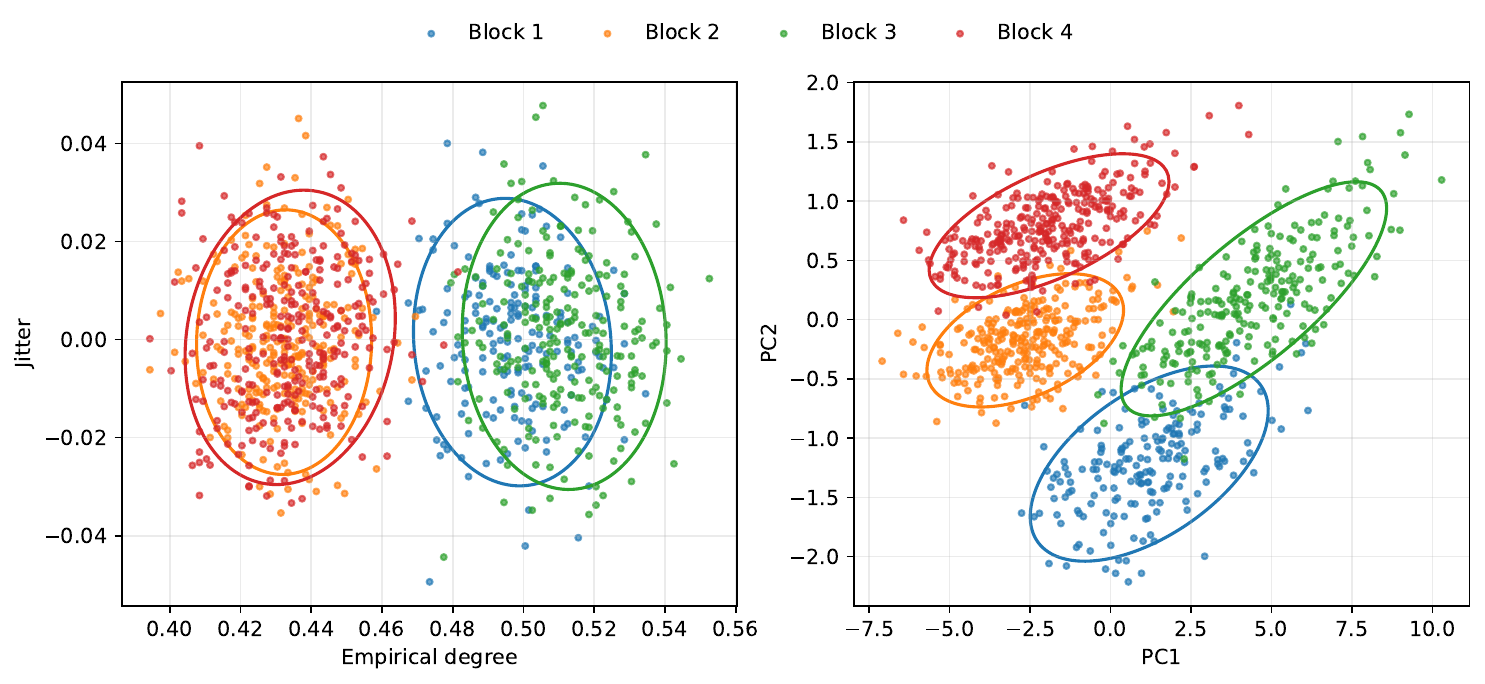}
			\caption{Broken-cycle-4}
		\end{subfigure}
		\hfill
		\begin{subfigure}{0.48\textwidth}
			\includegraphics[width=\linewidth]{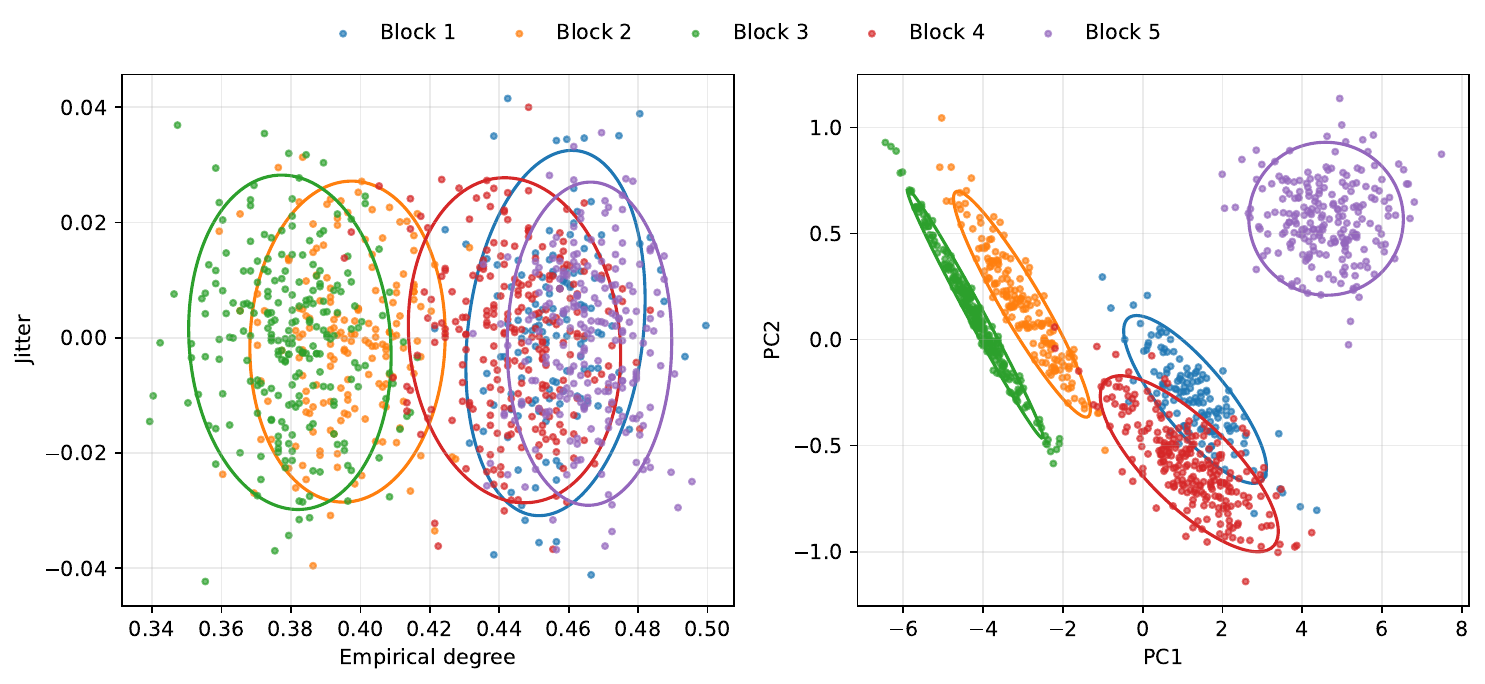}
			\caption{Gradient-5}
		\end{subfigure}
		
		\caption{Degree-only representation and PCA of rooted motif signatures for the SBM models.}
		\label{fig:degree-vs-motif-pca-sbm}
	\end{figure}

	\subsection{Experiment 2: Continuous graphons with degree plateaux}
	We next consider a continuous finite-rank graphon, a class closely related to spectral representations of graphon operators and low-rank network models \citep{Wolfe2013}. We  consider a graphon designed to exhibit plateau-like regions in the degree function while preserving non-trivial higher-order motif information, defined as
	\begin{equation}
		W_1(u,v)
		=
		0.35
		+
		0.10\,h(u)h(v)
		+
		0.22\sin(6\pi u)\sin(6\pi v),
	\end{equation}
	where
	\[
	h(u)
	=
	s\!\left(\frac{u}{0.3}\right)\mathbf 1_{\{u<0.3\}}
	+
	\mathbf 1_{\{0.3\le u\le 0.7\}}
	+
	\left[
	1+s\!\left(\frac{u-0.7}{0.3}\right)
	\right]\mathbf 1_{\{u>0.7\}},
	\qquad
	s(x)=3x^2-2x^3.
	\]
	
	For this graphon, we compute the population rooted motif curves associated with the family of selected rooted motifs \(\mathcal{F}\). These quantities are evaluated numerically through discretized graphon integral operators on a fine grid.
	
	Figures~\ref{fig:plateau-relay-curves} shows the resulting motif curves. The degree function exhibits extended regions where its variation is limited, suggesting that nodes located in those regions would be difficult to distinguish using degree information alone. In contrast, several higher-order motif coordinates vary substantially within the same intervals. 
	
	\begin{figure}[t]
		\centering
		\includegraphics[width=\textwidth]{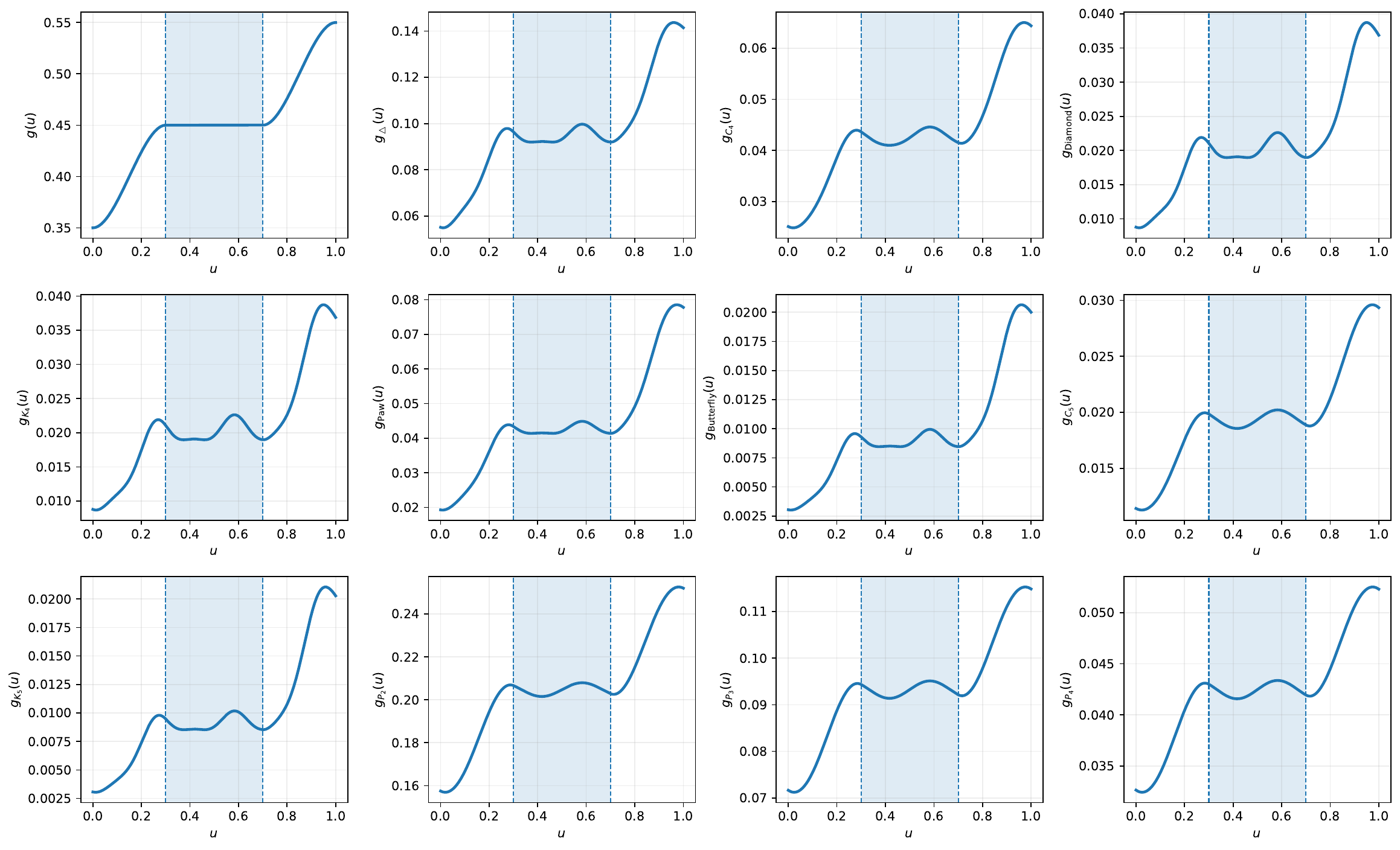}
		\caption{
			Population rooted motif curves associated with the graphon $W_1$.
			The highlighted interval corresponds to the plateau region of the degree profile.
		}
		\label{fig:plateau-relay-curves}
	\end{figure}

	\section{Discussion}
	
	This work proposes rooted motif signatures as a way to move beyond degree-based identifiability in graphon models. The main message is that local higher-order motif information can distinguish latent positions that degree alone cannot separate. In this sense, the identifiability based on rooted motif signatures provides a sufficient condition for constructing richer latent representations, while being less restrictive than strict degree monotonicity.
	
	For finite-rank graphons, our results show that signature-identifiability coincides with
	the necessary graphon identifiability condition of \cite{Borgs2010}. Indeed, under generic spectral assumptions, rooted motif signatures determine connectivity profiles and therefore separate all twin classes.
	The situation is more subtle for arbitrary graphons. Distinct latent positions may have
	identical rooted motif signatures even when their connectivity profiles are different.
	This can happen when the graphon is invariant under a nontrivial symmetry of the latent
	space. In such cases, signature-identifiability cannot coincide with the necessary
	graphon identifiability condition in full generality.
	
	Several extensions are natural. One direction is to use motif signatures for recovering latent groups, such as blocks in stochastic block models, especially in settings where degree information is uninformative. Another is to use motif-induced distances in graphon estimation procedures based on smoothing, nearest-neighbor averaging, or local aggregation. Rooted motif signatures could also be used as node features in downstream learning tasks, such as link
	prediction, node classification, or graph neural network architectures without external covariates. Finally, an important practical question is how to select the motif family and possibly learn motif weights, since different motifs may carry different information depending on the underlying graphon.

	% ===== BIBLIOGRAPHIE =====
	\bibliographystyle{apalike}
	%\bibliography{BiblioIdentifiability}
	\bibliography{bibliographie}
	
	\appendix
	\section{Proofs of theoretical results}	
	\subsection{Proof of Proposition~\ref{prop:finite-rank-path}}
	
	\begin{proof}
		Since \(W\) has rank \(m\), the associated integral operator satisfies
		\begin{align*}
			T_W f
			&=
			\sum_{j=1}^m
			\lambda_j
			\langle f,\phi_j\rangle
			\phi_j .
		\end{align*}
		Applying this identity to the constant function \(\mathbf 1\), we obtain
		\begin{align*}
			T_W\mathbf 1
			&=
			\sum_{j=1}^m
			\lambda_j
			\langle \mathbf 1,\phi_j\rangle
			\phi_j .
		\end{align*}
		Iterating the operator then gives, for every \(k\geq 1\),
		\begin{align*}
			T_W^k\mathbf 1
			&=
			T_W^{k-1}
			\left(
			\sum_{j=1}^m
			\lambda_j
			\langle \mathbf 1,\phi_j\rangle
			\phi_j
			\right) \\
			&=
			\sum_{j=1}^m
			\lambda_j
			\langle \mathbf 1,\phi_j\rangle
			T_W^{k-1}\phi_j \\
			&=
			\sum_{j=1}^m
			\lambda_j
			\langle \mathbf 1,\phi_j\rangle
			\lambda_j^{k-1}\phi_j \\
			&=
			\sum_{j=1}^m
			\lambda_j^k
			\langle \mathbf 1,\phi_j\rangle
			\phi_j .
		\end{align*}
		
		Now observe that the rooted path density on \(k+1\) nodes is given by
		\begin{align*}
			t(P_{k+1}^{\bullet},W)(u)
			&=
			\int_{[0,1]^k}
			W(u,x_1)W(x_1,x_2)\cdots W(x_{k-1},x_k)
			\,dx_1\cdots dx_k \\
			&=
			(T_W^k\mathbf 1)(u).
		\end{align*}
		Hence, for \(k=1,\ldots,m\),
		\begin{align*}
			t(P_{k+1}^{\bullet},W)(u)
			&=
			\sum_{j=1}^m
			\lambda_j^k
			\langle \mathbf 1,\phi_j\rangle
			\phi_j(u).
		\end{align*}
		Therefore the first \(m\) rooted path densities satisfy
		\begin{align*}
			\begin{pmatrix}
				t(P_2^\bullet,W)(u)\\
				t(P_3^\bullet,W)(u)\\
				\vdots\\
				t(P_{m+1}^\bullet,W)(u)
			\end{pmatrix}
			&=
			\begin{pmatrix}
				\lambda_1 c_1 & \lambda_2 c_2 & \cdots & \lambda_m c_m\\
				\lambda_1^2 c_1 & \lambda_2^2 c_2 & \cdots & \lambda_m^2 c_m\\
				\vdots & \vdots & & \vdots\\
				\lambda_1^m c_1 & \lambda_2^m c_2 & \cdots & \lambda_m^m c_m
			\end{pmatrix}
			\begin{pmatrix}
				\phi_1(u)\\
				\phi_2(u)\\
				\vdots\\
				\phi_m(u)
			\end{pmatrix},
		\end{align*}
		where
		\(
		c_j=\langle \mathbf 1,\phi_j\rangle .
		\)
		The matrix above can be written as
		\[
		V\,\operatorname{diag}(c_1,\ldots,c_m),
		\]
		where
		\[
		V=
		\begin{pmatrix}
			\lambda_1 & \lambda_2 & \cdots & \lambda_m\\
			\lambda_1^2 & \lambda_2^2 & \cdots & \lambda_m^2\\
			\vdots & \vdots & & \vdots\\
			\lambda_1^m & \lambda_2^m & \cdots & \lambda_m^m
		\end{pmatrix}
		\]
		is a Vandermonde matrix. Since the eigenvalues
		\(\lambda_1,\ldots,\lambda_m\) are pairwise distinct,
		\[
		\det(V)
		=
		\Bigl(\prod_{j=1}^m \lambda_j\Bigr)
		\prod_{1\le i<j\le m}
		(\lambda_j-\lambda_i)
		\neq 0.
		\]
		Therefore \(V\) is invertible. Moreover,
		\[
		\det\!\bigl(\operatorname{diag}(c_1,\ldots,c_m)\bigr)
		=
		\prod_{j=1}^m c_j
		\neq 0
		\]
		since \(c_j=\langle \mathbf 1,\phi_j\rangle\neq0\) for every \(j\). Hence
		\(
		V\,\operatorname{diag}(c_1,\ldots,c_m)
		\)
		is invertible. Hence, the rooted path densities
		\(
		t(P_k^\bullet,W)(u),
		\,
		k=2,\ldots,m+1,
		\)
		determine the spectral coordinates \(\Theta(u)\).
	\end{proof}

	\subsection{Proof of Theorem~\ref{thm:empirical-signature-concentration}}
	
	\begin{proof}	
		It is enough to prove the result for one fixed rooted motif \((F,r)\). Let
		\[
		m=|V(F)|,
		\qquad
		e=|E(F)|.
		\]
		For a node \(i\), write
		\[
		\widehat t_i(F)
		=
		\frac{1}{(n-1)_{m-1}}
		\sum_{\psi:V(F)\setminus\{r\}\hookrightarrow[n]\setminus\{i\}}
		\prod_{\{a,b\}\in E(F)}
		A_{\psi_i(a)\psi_i(b)}.
		\]
		Conditionally on
		\(
		\mathcal U=(U_1,\ldots,U_n),
		\)
		the conditional expectation of \(\widehat t_i(F)\) is
		\begin{align*}
			\bar t_i(F)
			&:=
			\mathbb E\!\left[
			\widehat t_i(F)
			\,\middle|\,
			\mathcal U
			\right] \\
			&=
			\frac{1}{(n-1)_{m-1}}
			\sum_{\psi}
			\prod_{\{a,b\}\in E(F)}
			W(U_{\psi_i(a)},U_{\psi_i(b)}).
		\end{align*}

		We decompose the error as
		\[
		\widehat t_i(F)-t((F,r),W)(U_i)
		=
		\Bigl(\widehat t_i(F)-\bar t_i(F)\Bigr)
		+
		\Bigl(\bar t_i(F)-t((F,r),W)(U_i)\Bigr).
		\]
		The first term is the edge-sampling error, conditionally on the latent variables \(\mathcal{U}\). Changing an edge incident to the root \(i\) can affect at most \(C_F n^{m-2}\)
		rooted embeddings, whereas changing an edge not incident to \(i\) can affect at most
		\(C_F n^{m-3}\) rooted embeddings. Since the normalizing factor is of order
		\(n^{m-1}\), the corresponding bounded-difference constants are respectively
		\(C_F/n\) and \(C_F/n^2\). For an edge \(e=(a,b)\), let
		\[
		c_e
		=
		\sup_{A,A'}
		\left|
		\widehat t_i(F;A)
		-
		\widehat t_i(F;A')
		\right|,
		\]
		where \(A\) and \(A'\) differ only in the entry corresponding to \(e\).Therefore
		\[
		\sum_e c_e^2
		\leq
		n\left(\frac{C_F}{n}\right)^2
		+
		n^2\left(\frac{C_F}{n^2}\right)^2
		\leq
		\frac{C_F'}{n}.
		\]
		McDiarmid's inequality \citep{McDiarmid1989} then yields
		\begin{equation}\label{eq:edge-sampling-bound}
			\mathbb P\left(
			\left|
			\widehat t_i(F)-\bar t_i(F)
			\right|>\eta
			\,\middle|\,
			\mathcal U
			\right)
			\leq
			2\exp(-c_F n\eta^2).
		\end{equation}
		We now control the second term,
		\(
		\bar t_i(F)-t((F,r),W)(U_i).
		\) This is the latent-sampling error.
		Conditionally on \(U_i\), the quantity \(\bar t_i(F)\) is a bounded \(U\)-statistic of
		order \(m-1\), based on the latent variables \(\{U_j:j\neq i\}\). Its kernel is
		\[
		h_{F,U_i}(x_1,\ldots,x_{m-1})
		=
		\prod_{\{a,b\}\in E(F)}
		W(z_a,z_b),
		\]
		where \(z_r=U_i\) and \(z_a=x_a\) for \(a\neq r\). Since \(0\leq W\leq 1\), this kernel
		is bounded by \(1\). Moreover,
		\begin{align*}
			\mathbb E\!\left[
			\bar t_i(F)
			\,\middle|\,
			U_i
			\right]
			&=
			\int_{[0,1]^{m-1}}
			\prod_{\{a,b\}\in E(F)}
			W(z_a,z_b)
			\prod_{a\neq r}dz_a  \\
			&=
			t((F,r),W)(U_i).
		\end{align*}
		Thus the second term is a centered bounded \(U\)-statistic. By the bounded-differences inequality for \(U\)-statistics \citep{Boucheron2013,Hoeffding1963,Arcones1995}, there exist constants
		\(C_F,c_F>0\), depending only on \(F\), such that for every \(\eta>0\),
		\begin{equation}
			\label{eq:latent-sampling-bound}
			\mathbb P\left(
			\left|
			\bar t_i(F)-t((F,r),W)(U_i)
			\right|>\eta
			\,\middle|\,
			U_i
			\right)
			\leq
			2\exp(-c_F n\eta^2).
		\end{equation}
		
		Combining the edge-sampling bound~\ref{eq:edge-sampling-bound} and the latent-sampling bound~\ref{eq:latent-sampling-bound}, we obtain, for every
		\(\eta>0\),
		\begin{equation}
			\label{eq:single-node-motif-bound}
			\mathbb P\left(
			\left|
			\widehat t_i(F)-t((F,r),W)(U_i)
			\right|>2\eta
			\right)
			\leq
			4\exp(-c_F n\eta^2).
		\end{equation}
		We now apply a union bound over all nodes and all motifs in
		\(
		\mathcal F=\{(F_1,r_1),\ldots,(F_K,r_K)\}.
		\)
		Let
		\[
		Z_{i,k}
		=
		\left|
		\widehat t_i(F_k)-t((F_k,r_k),W)(U_i)
		\right|.
		\]
		From \eqref{eq:single-node-motif-bound}, there exist constants \(C_0,c_0>0\), depending
		only on \(\mathcal F\), such that
		\[
		\mathbb P(Z_{i,k}>s)
		\leq
		C_0\exp(-c_0 n s^2)
		\]
		for every \(i=1,\ldots,n\), \(k=1,\ldots,K\), and \(s>0\). Therefore,
		\begin{align*}
			\mathbb P\left(
			\max_{1\leq i\leq n}
			\max_{1\leq k\leq K}
			Z_{i,k}
			>s
			\right)
			&\leq
			\sum_{i=1}^n
			\sum_{k=1}^K
			\mathbb P(Z_{i,k}>s) \\
			&\leq
			nK C_0\exp(-c_0 n s^2).
		\end{align*}
		
		For every \(\delta\in(0,1)\), choosing \(s= C
		\sqrt{
			\frac{\log(nK)+\log(1/\delta)}
			{n}
		}\), we have
		\[
		\mathbb P\left(
		\max_{1\leq i\leq n}
		\max_{1\leq k\leq K}
		\left|
		\widehat t_i(F_k)-t((F_k,r_k),W)(U_i)
		\right|
		>
		C
		\sqrt{
			\frac{\log(nK)+\log(1/\delta)}
			{n}
		}
		\right)
		\leq
		\delta .
		\]
		Since
		\[
		\left\|
		\widehat\Phi_i^{\mathcal F}
		-
		\Phi_W^{\mathcal F}(U_i)
		\right\|_2
		\leq
		\sqrt K
		\max_{1\leq k\leq K}
		\left|
		\widehat t_i(F_k)-t((F_k,r_k),W)(U_i)
		\right|,
		\]
		it follows that, with probability at least \(1-\delta\),
		\[
		\max_{1\leq i\leq n}
		\left\|
		\widehat\Phi_i^{\mathcal F}
		-
		\Phi_W^{\mathcal F}(U_i)
		\right\|_2
		\leq
		C
		\sqrt{
			\frac{\log(nK)+\log(1/\delta)}
			{n}
		}.
		\]
		This proves the theorem.
	\end{proof}

	\subsection{Proof of the corollary~\ref{cor:signature-distance}}
	
	\begin{proof}
		For \(i,j\in\{1,\ldots,n\}\), recall that
		\[
		\widehat d_{\mathcal F,w}(i,j)
		=
		\left\|
		\widehat\Phi_i^{\mathcal F}
		-
		\widehat\Phi_j^{\mathcal F}
		\right\|_w
		\]
		and
		\[
		d_{\mathcal F,w}(U_i,U_j)
		=
		\left\|
		\Phi_W^{\mathcal F}(U_i)
		-
		\Phi_W^{\mathcal F}(U_j)
		\right\|_w,
		\]
		where
		\[
		\|x\|_w
		=
		\left(
		\sum_{k=1}^K
		w_k x_k^2
		\right)^{1/2}.
		\]
		Therefore,
		\begin{align*}
			\left|
			\widehat d_{\mathcal F,w}(i,j)
			-
			d_{\mathcal F,w}(U_i,U_j)
			\right| 
			&=
			\left|
			\left\|
			\widehat\Phi_i^{\mathcal F}
			-
			\widehat\Phi_j^{\mathcal F}
			\right\|_w
			-
			\left\|
			\Phi_W^{\mathcal F}(U_i)
			-
			\Phi_W^{\mathcal F}(U_j)
			\right\|_w
			\right|.
		\end{align*}
		Using the reverse triangle inequality,
		\[
		\bigl|\|a\|_w-\|b\|_w\bigr|
		\leq
		\|a-b\|_w,
		\]
		with
		\[
		a=
		\widehat\Phi_i^{\mathcal F}
		-
		\widehat\Phi_j^{\mathcal F},
		\qquad
		b=
		\Phi_W^{\mathcal F}(U_i)
		-
		\Phi_W^{\mathcal F}(U_j),
		\]
		we obtain
		\begin{align*}
			\left|
			\widehat d_{\mathcal F,w}(i,j)
			-
			d_{\mathcal F,w}(U_i,U_j)
			\right| 
			&\leq
			\left\|
			\left(
			\widehat\Phi_i^{\mathcal F}
			-
			\widehat\Phi_j^{\mathcal F}
			\right)
			-
			\left(
			\Phi_W^{\mathcal F}(U_i)
			-
			\Phi_W^{\mathcal F}(U_j)
			\right)
			\right\|_w \\
			&=
			\left\|
			\left(
			\widehat\Phi_i^{\mathcal F}
			-
			\Phi_W^{\mathcal F}(U_i)
			\right)
			-
			\left(
			\widehat\Phi_j^{\mathcal F}
			-
			\Phi_W^{\mathcal F}(U_j)
			\right)
			\right\|_w \\
			&\leq
			\left\|
			\widehat\Phi_i^{\mathcal F}
			-
			\Phi_W^{\mathcal F}(U_i)
			\right\|_w
			+
			\left\|
			\widehat\Phi_j^{\mathcal F}
			-
			\Phi_W^{\mathcal F}(U_j)
			\right\|_w.
		\end{align*}
		Taking the maximum over \(i\) and \(j\), we get
		\begin{align*}
			\max_{1\leq i,j\leq n}
			\left|
			\widehat d_{\mathcal F,w}(i,j)
			-
			d_{\mathcal F,w}(U_i,U_j)
			\right| 
			&\leq
			2
			\max_{1\leq i\leq n}
			\left\|
			\widehat\Phi_i^{\mathcal F}
			-
			\Phi_W^{\mathcal F}(U_i)
			\right\|_w.
		\end{align*}
		Moreover, for every \(i\),
		\begin{align*}
			\left\|
			\widehat\Phi_i^{\mathcal F}
			-
			\Phi_W^{\mathcal F}(U_i)
			\right\|_w
			&=
			\left(
			\sum_{k=1}^K
			w_k
			\left[
			\widehat t_i(F_k)
			-
			t((F_k,r_k),W)(U_i)
			\right]^2
			\right)^{1/2} \\
			&\leq
			\left(
			\sum_{k=1}^K
			w_k
			\right)^{1/2}
			\max_{1\leq k\leq K}
			\left|
			\widehat t_i(F_k)
			-
			t((F_k,r_k),W)(U_i)
			\right|.
		\end{align*}
		By Theorem~\ref{thm:empirical-signature-concentration}, with probability at least
		\(1-\delta\),
		\[
		\max_{1\leq i\leq n}
		\max_{1\leq k\leq K}
		\left|
		\widehat t_i(F_k)
		-
		t((F_k,r_k),W)(U_i)
		\right|
		\leq
		C
		\sqrt{
			\frac{\log(nK)+\log(1/\delta)}
			{n}
		}.
		\]
		Hence, on the same event,
		\[
		\max_{1\leq i,j\leq n}
		\left|
		\widehat d_{\mathcal F,w}(i,j)
		-
		d_{\mathcal F,w}(U_i,U_j)
		\right|
		\leq
		2C
		\left(
		\sum_{k=1}^K
		w_k
		\right)^{1/2}
		\sqrt{
			\frac{\log(nK)+\log(1/\delta)}
			{n}
		}.
		\]
		In particular, if the weights are normalized so that
		\[
		\sum_{k=1}^K w_k=1,
		\]
		then
		\[
		\max_{1\leq i,j\leq n}
		\left|
		\widehat d_{\mathcal F,w}(i,j)
		-
		d_{\mathcal F,w}(U_i,U_j)
		\right|
		\leq
		2C
		\sqrt{
			\frac{\log(nK)+\log(1/\delta)}
			{n}
		}.
		\]
		This proves the high-probability bound.
	\end{proof}

	\section{Stochastic block model library}
	\label{app:sbm-library}
	
	This appendix reports the stochastic block models used in Table~\ref{tab:sbm-library}. The corresponding connectivity matrices are displayed in Figure~\ref{fig:sbm-library-heatmaps}. 
	\begin{figure}[ht]
		\centering
		\includegraphics[width=\textwidth]{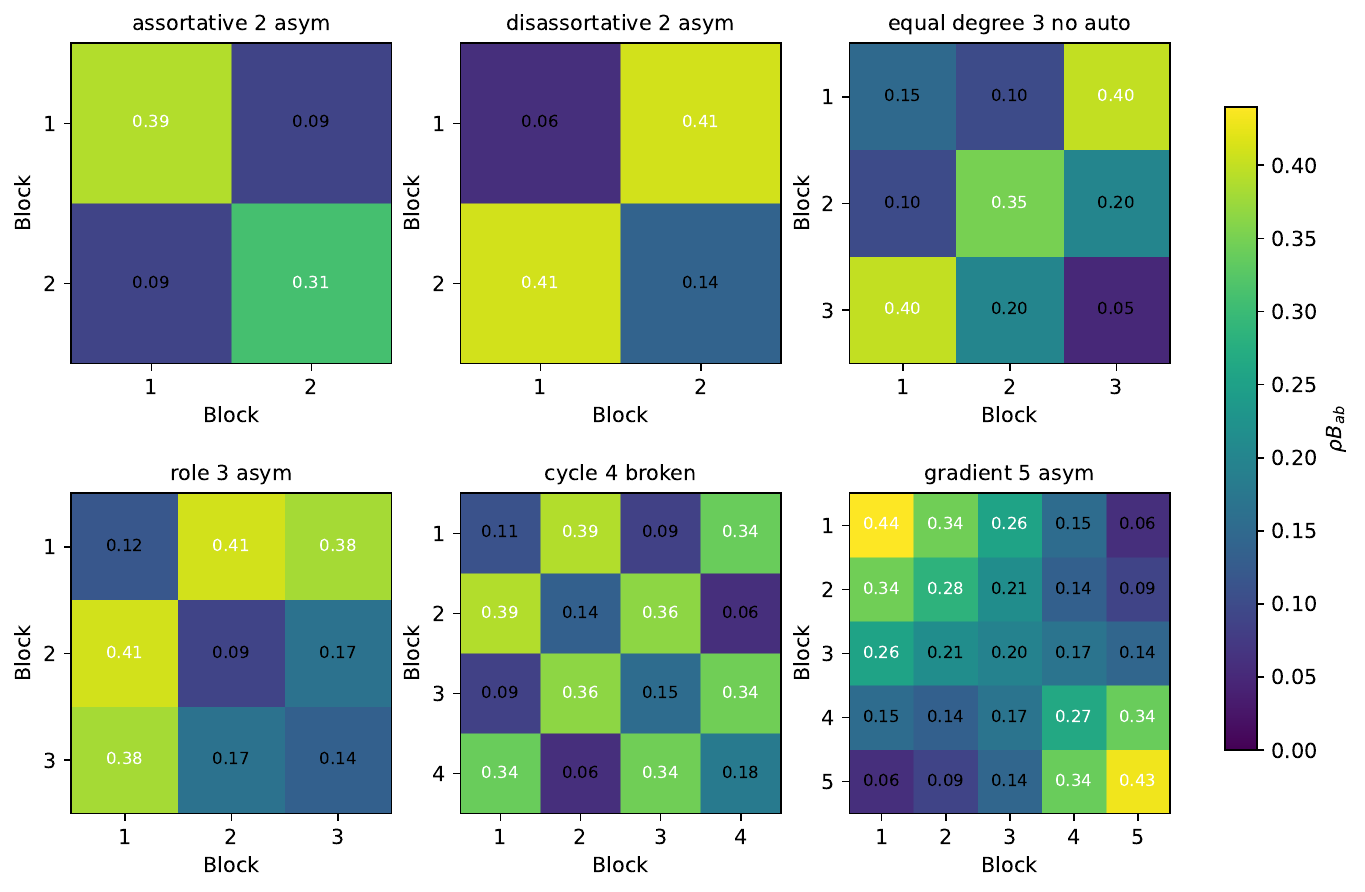}
		\caption{
			Connectivity matrices \(B\) defining the stochastic block model library used in Subsection~\ref{subsec:experiment1}.
		}
		\label{fig:sbm-library-heatmaps}
	\end{figure}

\end{document}